\documentclass[11pt]{article}
\usepackage{enumerate}
\usepackage{amssymb,a4wide,latexsym,makeidx,epsfig,fleqn}
\usepackage{amsthm}
\usepackage{amsmath}
\usepackage{graphicx}
\usepackage{float}
\usepackage{color}
\usepackage{indentfirst}
\usepackage{pdflscape}
\usepackage{siunitx} 
\usepackage{diagbox} 
\usepackage{rotating} 
\usepackage{booktabs}
\usepackage[colorlinks, linkcolor=blue, anchorcolor=blue, citecolor=blue]{hyperref}
\usepackage[numbers,sort&compress]{natbib}
\allowdisplaybreaks[4]
\newtheorem{theorem}{Theorem}[section]

\newtheorem{lemma}[theorem]{Lemma}

\newtheorem{corollary}[theorem]{Corollary}

\begin{document}
	\textwidth 150mm \textheight 225mm
	\title{On the distance signless Laplacian spectral radius, fractional
		matching and factors of graphs
		\thanks{Supported by the National Natural Science Foundation of China (No. 12271439).}}
	\author{{Zhenhao Zhang$^{a,b}$, Ligong Wang$^{a,b,}$\thanks{Corresponding author.}}\\
		{\small $^a$School of Mathematics and Statistics, Northwestern
			Polytechnical University,}\\ {\small  Xi'an, Shaanxi 710129,
			P.R. China.}\\
		{\small $^b$ Xi'an-Budapest Joint Research Center for Combinatorics, Northwestern
			Polytechnical University,}\\
		{\small Xi'an, Shaanxi 710129,
			P.R. China. }\\
		{\small E-mail: zhangzhenhao@mail.nwpu.edu.cn, lgwangmath@163.com} }
	\date{}
	\maketitle
	\begin{center}
		\begin{minipage}{120mm}
			\vskip 0.3cm
			\begin{center}
				{\small {\bf Abstract}}
			\end{center}
			{\small   The distance signless Laplacian matrix of a graph $G$ is defined as $Q(G)=$Tr$(G)+D(G)$, where Tr$(G)$ and $D(G)$ are  the diagonal matrix of vertex transmissions and the distance matrix of $G$, respectively. Denote by $E_G(v)$ the set of all edges incident to a vertex $v$ in $G$.  A fractional matching of a graph $G$ is a function $f:E(G) \rightarrow [0,1]$ such that $\sum_{e\in E_G(v)} f(e)\leq 1$ for every vertex $v\in V(G)$. The fractional matching number $\mu_f(G)$  of a graph $G$ is the maximum value of  $ \sum_{e\in E(G)} f(e)$ over all fractional matchings. Given subgraphs $H_1, H_2,...,H_k$ of $G$,  a $\{H_1, H_2,...,H_k\}$-factor of   $G$ is a spanning subgraph $F$ in which each connected component is isomorphic to one of $H_1, H_2,...,H_k$.
				In this paper, we establish an upper bound for the  distance signless Laplacian spectral radius of a graph $G$ of order $n$ to guarantee that $\mu_f(G)> \frac{n-k}{2}$, where $1\leq k<n$ is an integer. Besides, we also provide a sufficient condition based on distance signless Laplacian spectral radius to guarantee the existence of a $\{K_2,\{C_t\}\}$-factor in  a graph, where $t \geq 3$ is an integer.

				 }
		\end{minipage}
	\end{center}
	
	%%%%%%%%%%%%%%%%%%%%%%%%%%%%%%%%%%%%%%%%%%%%%%%%%%%%%%%%%%%%%%%%%%%%%%%%%%%%%%
	%%% Part 1: Introduction %%%
	\section{Introduction }
	\label{sec:ch6-introduction}
	% Para1: basic concept
	\noindent
	
	In this paper, all considered graphs are simple, connected and undirected. Let $G$ be a graph with vertex set $V(G)$ and edge set $E(G)$. The order of $G$ is $|V(G)|$, where $|\cdot|$ denotes the cardinality of a set. Let $N_G(v)$ be the set of all neighbours of a vertex $v$ in $G$. Let $d_G(v)=|N_G(v)|$ be the degree of a vertex $v$ in $G$. For a subset  $S \subseteq V(G)$, define $N_G(S)=\bigcup_{v\in S}N_G(v)$, and let $\delta(G)$ denote the minimum degree of $G$. If $d_G(v)=0$, then $v$ is an isolated vertex of $G$ .
	Denote by $I(G)$ the set of isolated vertices in $G$, and $i(G)=|I(G)|$. Let $E_G(v)=\{e \in E(G)\mid \text{$e$ is incident with  $v$ in $G$} \}$.
	Let $G$ and $H$ be two graphs. If $V(H)\subseteq V(G)$ and $E(H)\subseteq E(G)$, then $H$ is called a subgraph of $G$, denoted by $H \subseteq G$.
	Let $V_1 \subseteq V(G)$, $G-V_1$ is the graph formed from $G$ by deleting all
	the vertices in $V_1$ and their incident edges. For two graphs $G_1$ and $G_2$, we use $G_1+G_2$ to denote the union of $G_1$ and $G_2$ with vertex set $V(G_1)\cup V(G_2)$ and edge set $E(G_1)\cup E(G_2)$. The join $G_1\vee G_2$ is the graph obtained from vertex-disjoint graphs $G_1$ and $G_2$ by adding all possible edges between $V(G_1)$ and $V(G_2)$. The complement of $G$, denoted by $\overline{G}$, is a graph on the same vertex set $V$ such that two distinct vertices $u$ and $v$ are adjacent in $\overline{G}$ if and only if they are not adjacent in $G$.
	Let $K_n$, $C_n$ and $P_n$ denote the complete graph, the cycle and the path of order $n$, respectively.

	The distance between two vertices $v_i$ and $v_j$ of a graph $G$, denoted by $d_G(v_i,v_j)$ or $d_{i, j}$,  is the length of a shortest path connecting $v_i$ and $v_j$ in $G$. The distance matrix of $G$ is defined as $D(G)=(d_{i,j})$, is a symmetric real matrix. For $v_i \in V(G)$, the transmission of $v_i$ in $G$, denoted by $\text{Tr}_G(v_i)$, is the sum of distances from $v_{i}$ to all other vertices of $G$. The distance signless Laplacian matrix of $G$, denoted by $Q(G)$, is defined as $Q(G)=\text{Tr}(G)+D(G)$ where
	Tr$(G)= diag(\text{Tr}_G(v_1),\text{Tr}_G(v_2),...,\text{Tr}_G(v_n))$ is the diagonal matrix of the vertex transmissions in $G$.
	The distance signless Laplacian spectral radius of $G$, denoted by $q(G)$, is the largest eigenvalue of the nonnegative matrix $Q(G)$. The distance signless Laplacian matrix has been widely studied in spectral graph theory, as it encodes not only the distance information of a graph but also the transmission of each vertex. Its spectral radius is known to be sensitive to changes in graph structure, such as edge addition, vertex deletion, and graph operations, making it a useful tool for characterizing structural properties of graphs. He et al. \cite{He} gave some tight bounds on the signless Laplacian spectral radius with respect to the clique number. 
	Li, Fan and Zha \cite{Li1} gave a lower bound on the distance signless Laplacian spectral radius in terms of the chromatic number. Lin and Lu \cite{Lin} characterized the extremal graphs with the minimum and maximum distance signless Laplacian spectral radius among connected graphs with fixed clique number. 
	
	A fractional matching of a graph $G$ is a function $f:E(G) \rightarrow [0,1]$ such that $\sum_{e\in E_G(v)} f(e)$ $\leq 1$ for every vertex $v\in V(G)$. The fractional matching number of $G$, denoted by $\mu_f(G)$, is the maximum value of  $ \sum_{e\in E(G)} f(e)$ over all fractional matchings $f$ of $G$. Scheinerman and Ullman \cite{Sch} proposed a formula for calculating $\mu_f(G)$ of a graph $G$. Building on this result, Lou et al. \cite{Lou} derived a tight lower bound for the adjacency spectral radius to guarantee $\mu_f(G)>\frac{n-k}{2}$ in a graph $G$ with minimum degree $\delta$. Xu and Xi \cite{Xu} extended the above result to the distance spectral radius of a graph. Pan et al. \cite{Pan} established a relationship between the fractional matching number and the signless Laplacian spectral radius of a graph involving minimum degree. Li and Sun \cite{Li} established a sharp lower bound on the distance signless Laplacian spectral radius among all graphs of order $n$ with given matching number. Liu and Yan \cite{Liu} further refined this result. Yan et al. \cite{Yan} derived a lower bound of distance  Laplacian spectral radius among all graphs of order $n$ with given fractional matching number.

	Given subgraphs $H_1, H_2,...,H_k$ of $G$,  a $\{H_1, H_2,...,H_k\}$-factor of   $G$ is a spanning subgraph $F$ in which each connected component is isomorphic to one of $H_1, H_2,...,H_k$. Tutte \cite{Tut} established a necessary and sufficient condition for the existence of a $\{K_2,\{C_t\}\}$-factor in a graph from a structural perspective. Motivated by this result, Lou et al. \cite{Lou} gave a tight sufficient condition in terms of the spectral radius for the existence of a $\{K_2,\{C_t\}\}$-factor in a graph with minimum degree $\delta$. Xu and Xi \cite{Xu} extended the above result to the distance spectral radius of a graph.

	In this paper, we establish a tight upper bound of the distance signless Laplacian spectral radius to guarantee $\mu_f(G)>\frac{n-k}{2}$ in a graph $G$ of order $n$, where $1\leq k<n$ is an integer. We also provide a tight sufficient condition in terms of the distance signless Laplacian spectral radius for the existence of a $\{K_2,\{C_t\}\}$-factor in  a graph, where $t \geq 3$ is an integer.

	\section{Preliminaries}\label{sec:Preliminaries}
	
	In this section, we will give some useful concepts and lemmas which will be used later.
	
Let $M$ be a square matrix and let $\rho(M)$ be the spectral radius of $M$. Suppose $A=(a_{i,j})$ and $B=(b_{i,j})$ are two square matrices of order $n$. Define $A \leq B$ if $ a_{i,j} \leq b_{i,j}$ for any $ i,j \in \{1,2,...,n\}$.
	\begin{lemma}\label{lem1}
		(\cite{Ber})
		Suppose $A=(a_{i,j})$ and $B=(b_{i,j})$ are two nonnegative  square matrices of order $n$.  If $A \leq B$, then $\rho(A) \leq \rho(B)$. Furthermore, if $B$ is irreducible, $A \leq B$ and $A\not = B$, then $\rho(A)<\rho(B)$
	\end{lemma}
	
	Let $G$ be a graph with $e \not \in E(G)$.  The distance between any two vertices in $G+e$ is less  than or equal to those in $G$. Then $D(G+e)\leq D(G)$, Tr$(G+e)\leq $Tr$(G)$ and $Q(G+e)\leq Q(G)$.  The distance signless Laplacian matrix of a graph $G$ is nonnegative, so $q(G+e)\leq q(G)$.
	
	\begin{corollary}\label{cor1}
		
		Let	$G$  and $ H$ be two graphs and $e \not \in E(G)$. Then $q(G+e)\leq q(G)$. If $H \subseteq G$, then $q(G)\leq q(H)$.
	\end{corollary}
	
	Suppose $M$ is a square matrix of order $n$, $\mathcal{N}=\{1, 2, . . . , n\}$ and $\Pi:\mathcal{N}=\mathcal{N}_1\cup \mathcal{N}_2\cup . . . \cup \mathcal{N}_k$ is a partition of $\mathcal{N}$.
	The matrix $M$ can be partitioned into the following block form:
	\[M=\begin{pmatrix}
		M_{1, 1}&M_{1, 2}& . . . & M_{1, k} \\
		M_{2, 1}& M_{2, 2}& . . . &M_{2, k} \\
		\vdots&\vdots&\ddots&\vdots\\
		M_{k, 1}& M_{k, 2}& . . . &M_{k, k} \\
	\end{pmatrix}.\]
	
	The quotient matrix of $M$ with respect to $\Pi$, denoted by $B_{\Pi}=(b_{i, j})$, is the $k\times k$ matrix whose entry $b_{i, j}$ is the average row sums of the block $M_{i, j}$. The partition $\Pi$ is called equitable if each block $M_{i, j}$ of $M$ has constant row sum. In this case, the quotient matrix $B_{\Pi}$ is also referred to as equitable.

	\begin{lemma}\label{lem3}	(\cite{Bro, God})
		Let $M$ be a real symmetric matrix and $B_{\Pi}$ be an equitable quotient matrix of $M$. Then the eigenvalues of $B_{\Pi}$ are also eigenvalues of $M$. Furthermore, if $M$ is nonnegative and irreducible, then the largest eigenvalue of $B_{\Pi}$ is also the largest eigenvalue of $M$.	
	\end{lemma}
	
	\begin{lemma}\label{lem4}
		(\cite{Sch})
		Let $G$ be a graph of order $n$. Then
		\[\mu_f(G)=\frac{1}{2}(n-\max\{i(G-S)-|S|: \forall S\subseteq V(G)\}).\]
	\end{lemma}

	\begin{lemma}\label{lem5}
		(\cite{Tut, Lou})
		Let $G$ be a graph and let $t \geq 3$ be an integer. G has a $\{K_2,\{C_t\}\}$-factor if and only if $i(G-S)\leq |S|$ for every $ S\subseteq V(G)$.
		
	\end{lemma}

	\section{A Sufficient Condition on Distance Signless Laplacian Spectral Radius to guarantee  $\mu_f(G)> \frac{n-k}{2}$ }

	For a graph $G$ of order $n$, we establish a tight upper bound on the distance signless Laplacian spectral radius to ensure $\mu_f(G)>\frac{n-k}{2}$ in a graph $G$, where $1\leq k<n$ is an integer.

	\begin{theorem}\label{the1}
	
		Let $G$ be a graph of order $n\geq 14k+24$, where $1\leq k <n$ is an integer. If $q(G)\leq q(K_1 \vee (K_{n-2-k}+\overline{K_{1+k}}))$, then $\mu_f(G)>\frac{n-k}{2}$ unless $ G = K_1 \vee (K_{n-2-k}+\overline{K_{1+k}})$.
	\end{theorem}
	
	\begin{proof}
		Let $G$ be a graph of order $n\geq 14k+24$, $q(G)\leq q(K_1 \vee(K_{n-2-k}+\overline{K_{1+k}}))$ and $ G \not= K_1 \vee (K_{n-2-k}+\overline{K_{1+k}})$, where $1\leq k <n$ is an integer.
		Assume to the contrary that $\mu_f(G)\leq \frac{n-k}{2}$.
		By Lemma \ref{lem4} and $\mu_f(G)\leq \frac{n-k}{2}$, there exists a subset $S \subseteq V(G)$ so that  $i(G-S)-|S|\geq k$. Set $|S|=s$ and $n_1=n-s-i(G-S)$. Then $i(G-S)\geq s+k$ and $s\leq \frac{n-k}{2}$.
		
		If $n_1\geq 2$, then $V(G)$ can be partitioned into three disjoint subsets: $S$, $I(G-S)$ and $V(G)-S-I(G-S)$. Then we obtain that $G $ is a spanning subgraph of $K_s \vee (K_{n_1}+\overline{K_{i(G-S)}})$. If $i(G-S)> s+k$ for every $ S\subseteq V(G)$, then let $\widetilde{G}$ be the graph obtained from $G$ by adding all edges between $V(K_{n_1})$ and a vertex in $I(G-S)$, that is, $\widetilde{G}=K_s \vee (K_{n_1+1}+\overline{K_{i(G-S)-1}})$.
		Then $q(\widetilde{G})\leq q(G)$ and $i(\widetilde{G}-S)\geq s+k$. Repeat the above procedure until $i(G-S)= s+k$. Set $G_s=K_s \vee (K_{n-2s-k}+\overline{K_{s+k}})$.
		
		If {$n_1 =1$}, then  $V(K_{n_1})$ is a  isolated vertex after  deleting $S$. Therefore, $V(G)$ can be divided into $S$ and $I(G-S)$, it can be reduced to the case below.
		
		If {$n_1 =0$}, $V(G)$ can be divided into $S$ and $I(G-S)$, then $G$ is a spanning subgraph of $K_{n-i(G-S)}\vee \overline{K_{i(G-S)}}$. Similarly, we can add all edges between a vertex $v\in I(G-S)$ and $I(G-S)\setminus \{v\}$ when $i(G-S)> s+k+1$ for every $ S\subseteq V(G)$.
		Repeat the above procedure until $i(G'-S)= s+k $ or $s+k+1$, then $G'= K_\frac{n-k}{2} \vee \overline{K_\frac{n+k}{2}}=G_s$ where $s=\frac{n-k}{2}$ or $G'=K_{\frac{n-k-1}{2}} \vee \overline{K_{\frac{n+k+1}{2}}}= K_{\frac{n-k-1}{2}} \vee (K_1+\overline{K_{\frac{n+k-1}{2}}})=G_s$, where $s=\frac{n-k-1}{2}$. Hence $G$ is also a
		spanning subgraph of $G_s=K_s \vee (K_{n-2s-k}+\overline{K_{s+k}})$ with $s=\frac{n-k}{2}$ or $\frac{n-k-1}{2}$.

		Since $G$ is a spanning subgraph of $G_s=K_s \vee (K_{n-2s-k}+\overline{K_{s+k}})$ where $1\leq s \leq \frac{n-k}{2}$, we consider the following three cases.

		\par \noindent
		{\textbf{Case 1.}} $s=1$.
		
		In this case, we have $G_s = K_1 \vee (K_{n-2-k}+\overline{K_{1+k}})$.
		Since $q(K_1 \vee (K_{n-2-k}+\overline{K_{1+k}}))=q(G_s) \leq q(G) \leq q(K_1 \vee (K_{n-2-k}+\overline{K_{1+k}}))$ and $G\subseteq G_s$, we have $G = K_1 \vee (K_{n-2-k}+\overline{K_{1+k}})$ and $\mu_f(G)=\frac{n-k}{2}$ which contradicts $G\not  = K_1 \vee (K_{n-2-k}+\overline{K_{1+k}})$.

		\par \noindent{\textbf{Case 2.}} $1 < s<\frac{n-k}{2}$.
		
		In this case, we have $G_s=K_s \vee (K_{n-2s-k}+\overline{K_{s+k}})$ and the partition $V(G_s) = V(K_{n-2s-k}) \cup V({K_s})\cup V(\overline{K_{s+k}})$. The quotient matrix of $Q(G_s)$ is
		
		\[M_s=\begin{pmatrix}
			2n-s-2&s&2(s+k)\\
			n-2s-k&n+s-2&s+k\\
			2(n-2s-k)&s&2n+s+2k-4
		\end{pmatrix}.\]
	It can be  found that the characteristic polynomial of $M_s$ is
		\begin{align*}
			f_s(x)=&|xI_3-M_s|\\=&x^3+(8-5n-s-2k)x^2+(4k^2+2kn+12ks-8k+8n^2-ns-26n\\&+8s^2-4s+20)x-4k^2n-2k^2s+8k^2-12kns+4kn-4ks^2+26ks\\&-8k-4n^3+2n^2s+20n^2-8ns^2-2ns-32n-2s^3-18s^2-4s+16.
		\end{align*}
		Note that the characteristic polynomial of the quotient matrix of $Q(K_1 \vee (K_{n-2-k}+\overline{K_{1+k}}))$ is
		\begin{align*}
			\widetilde{f}(x)=&x^3+(7-5n-2k)x^2+(4k^2+2kn+4k+8n^2-27n+24)x\\&-2k^2(2n+1)+8k^2-8kn-14k-4n^3+22n^2-42n-8.
		\end{align*}
		Let $\theta_{s}$ and $\widetilde{\theta}$ be the largest roots of $f_s(x)=0$ and $\widetilde{f}(x)=0$, respectively. Since $G_s$ is  connected, then $Q(G_s)$ is  irreducible. By Lemma \ref{lem3}, we obtain $q(K_s \vee (K_{n-2s-k}+\overline{K_{s+k}}))=\theta_{s}$ and $q(K_1 \vee (K_{n-2-k}+\overline{K_{1+k}}))=\widetilde{\theta}$.
		
		Let $g_1(x)=f_s(x)-\widetilde{f}(x)$. Then
		\begin{align*}
			g_1(x)=&(1-s)[x^2+(n-12k-8s-4)x+2s^2+4(s+1)k+8n(s+1)\\&-16(1+s)+2k^2+12kn-26k-2n^2+2n+4].
		\end{align*}
		Notice that $q(G_s)\geq q(K_n)=2n-2$. Then we claim that
		$g_1(x)<0$ for $x \in [ 2n-2, +\infty )$.
		
		Let $g_1(x)=(1-s) \widehat{g_1}(x)$, where
		\begin{align*}
			\widehat{g}_1(x)=&x^2+(n-12k-8s-4)x+2s^2+4(s+1)k+8n(s+1)\\&-16(1+s)+2k^2+12kn-26k-2n^2+2n+4.
		\end{align*}
		As $n\geq 14k+24$ and $1< s<\frac{n-k}{2}$, we have $2n-2-(-\frac{n-12k-8s-4}{2})=\frac{5n}{2}-6k-6-4s \geq \frac{n}{2}-4k-6\geq 0$, which implies that  $\widehat{g_1}(x)$ is increasing on the interval $ [ 2n-2, +\infty )$. Let
		\begin{align*}
			h_1(s)=\widehat{g_1}(2n-2)&=2s^2+(4k-8n)s+10n-22k-18+12kn-(2n-2)(12k-n+4)\\&+(2n-2)^2+2k^2-2n^2+6. \end{align*}
		Since
		\begin{align*}
			\frac{dh_1(s)}{ds}&=4k-8n+4s \leq 2k-6n<0,
		\end{align*}
		we have $h_1(s)$ is monotonically decreasing for  $s$, and thus
		\begin{align*}
			h_1(s) \geq 	h_1\left(\frac{n-k}{2}\right)&=\frac{n^2}{2}-n(7k+8)+\frac{k^2}{2}+2k\\&=n\left(\frac{n}{2}-7k-8\right)+\frac{k^2}{2}+2k>0.
		\end{align*}
		Hence, $\widehat{g_1}(2n-2)\geq 0$ and $\widehat{g_1}(x)$ is increasing with respect to $x\in  [ 2n-2, +\infty )$, which implies that $g_1(x)=(1-s) \widehat{g_1}(x)<0$ when $x\in  [ 2n-2, +\infty )$,  as claimed. Since $g_1(x)<0$ for $x \in [ 2n-2, +\infty )$, we have $f_s(\widetilde{\theta})=f_s(\widetilde{\theta})-\widetilde{f}(\widetilde{\theta})=g_1(\widetilde{\theta})<0$, which implies that $\widetilde{\theta}<\theta_{s}$. This  contradicts the assumption $q(G_s)\leq q(G) \leq q(K_1 \vee (K_{n-2-k}+\overline{K_{1+k}}))$.
		
		\par\noindent{\textbf{Case 3.}} $s=\frac{n-k}{2}$.
		
		In this case, we have $G_\frac{n-k}{2}=K_\frac{n-k}{2} \vee \overline{K_\frac{n+k}{2}}$, and we have the partition $V(G) = V(K_\frac{n-k}{2}) \cup V(\overline{K_\frac{n+k}{2}})$. The quotient matrix of $Q(G_\frac{n-k}{2})$ is
		\[\widehat{M}=\begin{pmatrix}
			\frac{3n-k}{2}-2&\frac{n+k}{2}\\
			\frac{n-k}{2}&\frac{5n+3k}{2}-4\\
		\end{pmatrix}.\]
		It is clear that the characteristic polynomial of $\widehat{M}$ is
		\begin{align*}
			\widehat{f}(x)=x^2+(6-k-4n)x+\frac{7}{2}n^2-\frac{1}{2}k^2+kn-11n-k+8.
		\end{align*}
		In fact,
		\begin{align*}
			\widehat{f}(x)&=\left(x-\frac{3n-k}{2}+2\right)\left(x-\frac{5n+3k}{2}+4\right)-\left(\frac{n+k}{2}\right)\left(\frac{n-k}{2}\right)\\&<\left(x-\frac{3n-k}{2}+2\right)\left(x-\frac{5n+3k}{2}+4\right).
		\end{align*}
		Then the largest root of $\widehat{f}(x)$ is greater than $\frac{5n+3k}{2}-4$. We have
		
		\begin{align*} \widetilde{f}\left(\frac{5n+3k}{2}-4\right)&=\frac{3}{8}n^3-\left(\frac{39}{8}k+\frac{35}{4}\right)n^2+\left(-\frac{3}{8}k^2+16k+26\right)n+\frac{39}{8}k^3+\frac{35}{4}k^2-10k-20
			\\
			&\geq
			\frac{3}{8}n^3-\left(\frac{39}{8}k+\frac{35}{4}+\frac{3}{8}k\right)n^2+\left(16k+26\right)n+\frac{39}{8}k^3+\frac{35}{4}k^2-10k-20
			\\
			&=n^2\left(\frac{3}{8}n-\frac{42}{8}k+\frac{35}{4}\right)+\frac{39}{8}k^3+\frac{35}{4}k^2+(16n-10)k+(26n-20)>0.\\
			\widetilde{f}'(x)&=3x^2-2(5n+2k-7)x+8n^2+4k^2+2kn-27n+4k+24.
		\end{align*}
		$\widetilde{f}'(x)$ is a quadratic function opening upward. Its axis of symmetry is $x=\frac{5n+2k-7}{3}$.
		\begin{align*}
			\widetilde{f}'\left(\frac{5n+3k}{2}-4\right)=\frac{7}{4}n^2-\left(\frac{k}{2}+12\right)n+\frac{19}{4}k^2+5k+16>0.
		\end{align*}
		Since $\frac{5n+3k}{2}-4>\frac{5n+2k-7}{3}$, we have $ \widetilde{f}'(x)>0 $ for $x\in  [ \frac{5n+3k}{2}-4, +\infty )$. $ \widetilde{f}(x)$  is increasing with respect to $x\in  [ \frac{5n+3k}{2}-4, +\infty )$.
		\[
		\begin{cases}
			\widetilde{f}\left(\frac{5n+3k}{2}-4\right) &\geq 0. \\
			\widetilde{f}'(x) &\geq 0,  \qquad x\in  [ \frac{5n+3k}{2}-4, +\infty ).
		\end{cases}
		\]
		So $\widetilde{\theta} < \frac{5n+3k}{2}-4< \widehat{\theta}$, where $\widehat{\theta}$ is the largest root of $\widehat{f}(x)$.  It  contradicts the assumption $q(G_s)\leq  q(G)\leq q(K_1 \vee (K_{n-2-k}+\overline{K_{1+k}}))$.  
		
		This completes the proof.
	\end{proof}
	
	%%% Pro2%%%
	%%%%%%%%%%%%%%%%%%%%%%%%%%%%%%%%%%%%%%%%%%%%%%%%%%%%%%%%%%%%%%%%%%%%%%%%%%%%%%
	%%%%%%%%%%%%%%%%%%%%%%%%%%%%%%%%%%%%%%%%%%%%%%%%%%%%%%%%%%%%%%%%%%%%%%%%%%%%%%
	
	\section{A Sufficient Condition on Distance Signless Laplacian Spectral Radius for the existence of a $\{K_2,\{C_t\}\}$-factor}

	For a graph $G$ of order $n$, we establish a tight sufficient condition in terms of the distance signless Laplacian spectral radius for the existence of a $\{K_2,\{C_t\}\}$-factor in a graph, where $t \geq 3$ is an integer.

	\begin{theorem}\label{the2}
		Let $G$ be a graph of order $n$ and $t \geq 3$ be a positive integer.
		
		(1) If $n\geq 16$, $n=12,14$, and $q(G)\leq q(K_1 \vee (K_{n-3}+\overline{K_2}))$, then $G$  has a $\{K_2,\{C_t\}\}$-factor unless $ G = K_1 \vee (K_{n-3}+\overline{K_2})$.
		
		(2) If $3 \leq n \leq 11$, $n=13,15$, and $q(G)\leq q(\widehat{G})$, then $G$  has a $\{K_2,\{C_t\}\}$-factor unless $G = \widehat{G}$, where $\widehat{G}=K_\frac{n-1}{2} \vee (\overline{K_\frac{n+1}{2}})$ for odd $n$ and  $\widehat{G}=K_{\frac{n}{2}-1} \vee (\overline{K_{\frac{n}{2}+1}})$ for even $n$.
		
	\end{theorem}

	\begin{proof}
		The case for $n=3$ is straightforward to analyze, our analysis primarily focuses on the two scenarios where
	$4\leq n\leq 36$ and ${n\geq 37}$.
		 When $n=3$,
		$\widehat{G}=K_1 \vee \overline{K_2}=P_3$,
		the only $G$ satisfying $q(G)\leq q(\widehat{G})$ and $G\not = \widehat{G}$ must be $C_3$, which has a $\{K_2,\{C_t\}\}$-factor.
		
		Let us start with ${n\geq 37}$. Let $G$ be a graph of order $n\geq 37$,  $q(G)\leq q(K_1 \vee (K_{n-3}+\overline{K_{2}}))$ and $G \not = K_1 \vee (K_{n-3}+\overline{K_{2}})$. Assume to the contrary that $G$  doesn't have a $\{K_2,\{C_t\}\}$-factor. By Lemma \ref{lem5}, there exists $S \subseteq V(G)$ so that  $i(G-S)-|S|>0$. Set $|S|=s$, then $i(G-S)\geq s+1$. 
		Hence, $G$ must be a spanning subgraph of one of the following three graphs: $G_s=K_s \vee (K_{n-2s-1}+\overline{K_{s+1}})$ for $1 \leq s < \frac{n-1}{2}$; $\widehat{G}=K_{\frac{n-1}{2}} \vee \overline{K_{\frac{n+1}{2}}}$ for odd $n$; $\widehat{G}=K_{\frac{n}{2}-1} \vee \overline{K_{\frac{n}{2}+1}}$ for even $n$. Note that $K_{\frac{n}{2}-1} \vee \overline{K_{\frac{n}{2}+1}}= K_{\frac{n}{2}-1} \vee (K_1+\overline{K_{\frac{n}{2}}})=G_{\frac{n}{2}-1}$. Therefore, we only need to discuss the first two graphs $G_s=K_s \vee (K_{n-2s-1}+\overline{K_{s+1}})$ for $1\leq s<\frac{n-1}{2}$ and $G_{\frac{n-1}{2}}=K_{\frac{n-1}{2}} \vee \overline{K_{\frac{n+1}{2}}}$ for odd $n$.
		
		\par\noindent{\textbf{Case 1.}} $s=1$.
		
		In this case, we have $G_s = K_1 \vee (K_{n-3}+\overline{K_{2}})$.
		Since $q(K_1 \vee (K_{n-3}+\overline{K_{2}}))=q(G_s) \leq q(G) \leq q(K_1 \vee (K_{n-3}+\overline{K_{2}}))$ and $G \subseteq G_s$, we have $G = K_1 \vee (K_{n-3}+\overline{K_{2}})$ and $G$  has a $\{K_2,\{C_t\}\}$-factor, which contradicts that $G\not  = K_1 \vee (K_{n-3}+\overline{K_{2}})$.
		\par\noindent{\textbf{Case 2.}} $1< s<\frac{n-1}{2}$.
		
		In this case, we have $G_s=K_s \vee (K_{n-2s-1}+\overline{K_{s+1}})$ and the partition $V(G_s) = V(K_{n-2s-1}) \cup V({K_s})\cup V(\overline{K_{s+1}})$. The  quotient matrix of $Q(G_s)$ is
		\[M_s=\begin{pmatrix}
			2n-s-2&s&2(s+1)\\
			n-2s-1&n+s-2&s+1\\
			2(n-2s-1)&s&2n+s-2
		\end{pmatrix}.\]
		It is clear that the characteristic polynomial of $M_s$ is
		\begin{align}
			f_s(x)=&x^3+(6-5n-s)x^2+(8n^2-ns-24n+8s^2+8s+16)x+20s- \label{eq1}\\
			&32n-14ns-8ns^2+2ns^2+20n^2-4n^3+14s^2-2s^3+16.\notag
		\end{align}
		Note that the characteristic polynomial of the quotient matrix of $Q(K_1 \vee (K_{n-3}+\overline{K_{2}}))$ is
		\begin{align}
			\widetilde{f}(x)=&x^3+(5-5n)x^2+(8n^2-25n+32)x-4n^3+22n^2-52n+48.\label{eq2}
		\end{align}
		Let $\theta_{s}$ and $\widetilde{\theta}$ be the largest roots of $f_s(x)=0$ and $\widetilde{f}(x)=0$, respectively. Since $G_s$ is connected, then $Q(G_s)$ is  irreducible. By Lemma \ref{lem3}, we obtain $q(K_s \vee (K_{n-2s-1}+\overline{K_{s+1}}))=\theta_{s}$ and $q(K_1 \vee (K_{n-3}+\overline{K_{2}}))=\widetilde{\theta}$.
		
		Let $g_1(x)=f_s(x)-\widetilde{f}(x)$. Then $g_1(x)=(1-s)[x^2+(n-8s-16)x-2n^2+8ns+22n+2s^2-12s-32]$.  Notice that $q(G_s)\geq q(K_n)=2n-2$. Then we claim that	$g_1(x)<0$ for $x \in [ 2n-2, +\infty )$.
		
		Let $g_1(x)=(1-s) \widehat{g_1}(x)$, where
		$\widehat{g}_1(x)=x^2+(n-8s-16)x-2n^2+8ns+22n+2s^2-12s-32$.
		As $n\geq 37$ and $1< s<\frac{n-1}{2}$, we have $2n-2-(-\frac{n-8s-16}{2})=\frac{5n}{2}-12-4s \geq \frac{n}{2}-10\geq 0$, which implies that  $\widehat{g_1}(x)$ is increasing on the interval $ [ 2n-2, +\infty )$. Let $h_1(s)=\widehat{g_1}(2n-2)=2s^2+(4-8n)s+22+(2n-2)(n-16)+(2n-2)^2-2n^2-32$.
		
		Since $\frac{dh_1(s)}{ds}=4-8n+4s \leq 2-6n<0$. We have $h_1(s)$ is monotonically decreasing for  $s$, and thus
		$h_1(\frac{n-1}{2})=\frac{n^2}{2}-15n+\frac{5}{2}=(\frac{n}{2}-15)n+\frac{5}{2}>0$.
		
		Hence, $\widehat{g_1}(2n-2)\geq 0$ and $\widehat{g_1}(x)$ is increasing with respect to $x\in  [ 2n-2, +\infty )$, which implies that $g_1(x)=(1-s) \widehat{g_1}(x)<0$ when $x\in  [ 2n-2, +\infty )$,  as claimed. Since $g_1(x)<0$ for $x \in [ 2n-2, +\infty )$, we have $f_s(\widetilde{\theta})=f_s(\widetilde{\theta})-\widetilde{f}(\widetilde{\theta})=g_1(\widetilde{\theta})<0$, which implies that $\widetilde{\theta}<\theta_{s}$. This  contradicts the assumption $q(G_s)\leq q(G) \leq q(K_1 \vee (K_{n-3}+\overline{K_{2}}))$.
		
		\par\noindent{\textbf{Case 3.}} $s=\frac{n-1}{2}$.
		
		In this case, we have $G_\frac{n-1}{2}=K_\frac{n-1}{2} \vee \overline{K_\frac{n+1}{2}}$, and we have the partition $V(G) = V(K_\frac{n-1}{2}) \cup V(\overline{K_\frac{n+1}{2}})$. The quotient matrix of $Q(G_\frac{n-1}{2})$ is
		\[\widehat{M}=\begin{pmatrix}
			\frac{3n-5}{2}&\frac{n+1}{2}\\
			\frac{n-1}{2}&\frac{5n-5}{2}\\
		\end{pmatrix}.\]
		It is clear that the characteristic polynomial of $\widehat{M}$ is
		\begin{align}\label{eq3}
			\widehat{f}(x)=x^2+(5-4n)x+\frac{7}{2}n^2+\frac{13}{2}-10n.
		\end{align}
		In fact,
		\begin{align*}
			\widehat{f}(x)&=\left(x-\frac{3n-5}{2}\right)\left(x-\frac{5n-5}{2}\right)-\left(\frac{n+1}{2}\right)\left(\frac{n-1}{2}\right)\\&<\left(x-\frac{3n-5}{2}\right)\left(x-\frac{5n-5}{2}\right).
		\end{align*}
		Then the largest root of $\widehat{f}(x)$ is larger than $\frac{5n-5}{2}$.  We have
		
		\begin{align*} \widetilde{f}\left(\frac{5n-5}{2}\right)&=\frac{3}{8}n^3-\frac{109}{8}n^2+\frac{333}{8}n-\frac{131}{8}
			\\&=\frac{n^2}{8}(3n-109)+\frac{1}{8}(333n-131) >0 \quad (n\geq 37).
			\\
			\widetilde{f}'(x)&=3x^2-10(n-1)x+8n^2-25n+32.
		\end{align*}
		$\widetilde{f}'(x)$ is a quadratic function opening upward. Its axis of symmetry is $x=\frac{5n-5}{3}$.
		\begin{align*}
			\widetilde{f}'\left(\frac{5n-5}{2}\right)&=\frac{7}{4}n^2-\frac{25}{2}n+\frac{103}{4}\\
			&=\frac{n}{4}(7n-50)+\frac{103}{4} >0
			\quad (n\geq 37).
		\end{align*}
		Since $\frac{5n-5}{2}>\frac{5n-5}{3}$, we have $ \widetilde{f}'(x)>0 $ for $x\in  [ \frac{5n-5}{2}, +\infty )$. $ \widetilde{f}(x)$  is increasing with respect to $x\in  [ \frac{5n-5}{2}, +\infty )$.
		\[
		\begin{cases}
			\widetilde{f}\left(\frac{5n-5}{2}\right) &\geq 0. \\
			\widetilde{f}'(x) &\geq 0,   \quad x\in  [ \frac{5n-5}{2}, +\infty ).
		\end{cases}
		\]
		So $\widetilde{\theta}< \frac{5n-5}{2}< \widehat{\theta}$, where $\widehat{\theta}$ is the largest root of $\widehat{f}(x)$. It contradicts the assumption $q(G_s) \leq q(G)\leq q(K_1 \vee (K_{n-3}+\overline{K_{2}}))$.
		
		\vspace{0.2cm} 
		
		Now we need to analyze the case $4\leq n\leq 36$ for $G_s=K_s \vee (K_{n-2s-1}+\overline{K_{s+1}})$. Let $\widehat{G}=K_\frac{n-1}{2} \vee \overline{K_\frac{n+1}{2}}$ for odd $n$ and  $\widehat{G}=K_{\frac{n}{2}-1} \vee \overline{K_{\frac{n}{2}+1}}$ for even $n$. In fact, $\widehat{G}=G_{\lfloor\frac{n-1}{2}\rfloor}$.

		When ${4\leq n\leq 36}$, we don't know how $q(G_s)$ changes with $s$, nor do we know which one is larger between $q(G_s)$ and  $q(\widehat{G})$. However, we can calculate their distance signless Laplacian spectral radius by \eqref{eq1}, \eqref{eq2} and \eqref{eq3}. Calculation results are presented in Table \ref{example} below.

		By comparing the above data, we can conclude that either $q(G) > q(K_1 \vee (K_{n-3}+\overline{K_2}))$ or $G=K_1 \vee (K_{n-3}+\overline{K_2})$ for $n=12$ , $n=14$ and $16\leq n\leq 36$. This  contradicts the assumption. When $4 \leq n \leq 11$, $n=13$ and $n=15$, either $q(G) >q(\widehat{G})$ or $G=\widehat{G}$, which also contradicts the assumption.
		Combining the cases of $n=3$ and $n \geq 37$, we have now completed the proof of the entire Theorem \ref{the2}.

		\begin{sidewaystable}[htbp]
			
			\tiny{
				%\footnotesize
				%	\centering
				\caption{Distance signless Laplacian spectral radius of $G_s=K_s \vee (K_{n-2s-1}+\overline{K_{s+1}})$}\label{example}
				%%\begin{tabular}{|c|*{19}{|c}|}
				\begin{tabular}{|c|*{19}{|c}|}
					\toprule
					$q(\widehat{G})$ &\diagbox{$n$}{$q(G_s)$} & $q(G_1)$&$q(G_2)$&$q(G_3)$&$q(G_4)$&  $q(G_5)$ & $q(G_6)$	& $q(G_7)$ & $q(G_8)$ & $q(G_9)$ & $q(G_{10})$ & $q(G_{11})$ & $q(G_{12})$ & $q(G_{13})$ & $q(G_{14})$ & $q(G_{15})$ & $q(G_{16})$ & $q(G_{17})$ \\
					\midrule
					9.46 &$n=4$& 9.46 & & & & & & & & & & & & & & &   &\\
					11 &$n=5$& 12.58 & 11 & & & & & & & & & & & & & & & \\
					14.90 &$n=6$& 15.46 & 14.90 & & & & & & & & & & & & & & &\\
					16.42 &$n=7$& 18.21 & 18.18 & 16.42 & & & & & & & & & & & & & &\\
					20.32 &$n=8$& 20.87 & 21.21 & 20.32 & & & & & & & & & & & & & &\\
					21.84 &$n=9$& 23.46 & 24.10 & 23.71 & 21.84 & & & & & & & & & & & & &\\
					25.75 &$n=10$& 26.01 & 26.89 & 26.85 & 25.75 & & & & & & & & & & & & &\\
					27.26 &$n=11$& 28.53 & 29.61 & 29.85 & 29.21 & 27.26 & & & & & & & & & & & & \\
					31.17 &$n=12$& 31.01 & 32.28 & 32.74 & 32.43 & 31.17 & & & & & & & & & & & & \\
					32.68 &$n=13$& 33.47 & 34.90 & 35.56 & 35.51 & 34.68 & 32.68 & & & & & & & & & & & \\
					36.58 &$n=14$& 35.90 & 37.49 & 38.32 & 38.49 & 37.98 & 36.58 & & & & & & & & & & & \\
					38.09 &$n=15$& 38.32 & 40.04 & 41.02 & 41.39 & 41.13 & 40.14 & 38.09 & & & & & & & & & & \\
					42 &$n=16$& 40.72 & 42.57 & 43.69 & 44.22 & 44.17 & 43.49 & 42 & & & & & & & & & & \\
					43.51 &$n=17$& 43.10 & 45.07 & 46.33 & 47 & 47.13 & 46.70 & 45.59 & 43.51 & & & & & & & & & \\
					47.42 &$n=18$& 45.47 & 47.56 & 48.93 & 49.74 & 50.03 & 49.80 & 48.99 & 47.42 & & & & & & & & & \\
					48.93 &$n=19$& 47.83 & 50.03 & 51.51 & 52.44 & 52.87 & 52.82 & 52.25 & 51.04 & 48.93 & & & & & & & & \\
					52.83 &$n=20$& 50.18 & 52.48 & 54.07 & 55.11 & 55.67 & 55.78 & 55.40 & 54.47 & 52.83 & & & & & & & & \\
					54.34 &$n=21$& 52.52 & 54.92 & 56.61 & 57.75 & 58.43 & 58.67 & 58.47 & 57.77 & 56.48 & 54.34 & & & & & & & \\
					58.25 &$n=22$& 54.84 & 57.34 & 59.12 & 60.37 & 61.16 & 61.53 & 61.47 & 60.97 & 59.95 & 58.25 & & & & & & & \\
					59.76 &$n=23$& 57.16 & 59.75 & 61.63 & 62.97 & 63.86 & 64.34 & 64.42 & 64.09 & 63.28 & 61.91 & 59.76 & & & & & & \\
					63.66 &$n=24$& 59.48 & 62.15 & 64.11 & 65.54 & 66.53 & 67.12 & 67.32 & 67.13 & 66.52 & 65.41 & 63.66 & & & & & & \\
					65.17 &$n=25$& 61.78 & 64.54 & 66.59 & 68.10 & 69.17 & 69.86 & 70.18 & 70.12 & 69.67 & 68.78 & 67.35 & 65.17 & & & & & \\
					69.08 &$n=26$& 64.08 & 66.92 & 69.05 & 70.64 & 71.80 & 72.58 & 73 & 73.07 & 72.76 & 72.05 & 70.87 & 69.08 & & & & & \\
					70.59 &$n=27$& 66.37 & 69.30 & 71.49 & 73.17 & 74.41 & 75.27 & 75.79 & 75.97 & 75.79 & 75.24 & 74.27 & 72.78 & 70.59 & & & & \\
					74.49 &$n=28$& 68.65 & 71.66 & 73.93 & 75.68 & 77.00 & 77.95 & 78.55 & 78.83 & 78.77 & 78.37 & 77.57 & 76.32 & 74.49 & & & & \\
					76 &$n=29$& 70.93 & 74.02 & 76.36 & 78.18 & 79.57 & 80.60 & 81.29 & 81.66 & 81.71 & 81.43 & 80.79 & 79.75 & 78.21 & 76 & & & \\
					79.91 &$n=30$& 73.21 & 76.36 & 78.78 & 80.66 & 82.13 & 83.23 & 84 & 84.46 & 84.61 & 84.45 & 83.95 & 83.08 & 81.77 & 79.91 & & & \\
					81.41 &$n=31$& 75.48 & 78.71 & 81.19 & 83.14 & 84.67 & 85.84 & 86.69 & 87.23 & 87.48 & 87.42 & 87.05 & 86.33 & 85.22 & 83.63 & 81.41 & & \\
					85.32 &$n=32$& 77.74 & 81.04 & 83.59 & 85.61 & 87.20 & 88.44 & 89.36 & 89.98 & 90.31 & 90.35 & 90.10 & 89.52 & 88.58 & 87.21 & 85.32 & & \\
					86.83 &$n=33$& 80.00 & 83.37 & 85.98 & 88.06 & 89.72 & 91.03 & 92.02 & 92.71 & 93.12 & 93.26 & 93.10 & 92.64 & 91.85 & 90.69 & 89.06 & 86.83 & \\
					90.74 &$n=34$& 82.26 & 85.69 & 88.37 & 90.51 & 92.23 & 93.60 & 94.65 & 95.42 & 95.91 & 96.13 & 96.07 & 95.72 & 95.07 & 94.07 & 92.66 & 90.74 & \\
					92.24 &$n=35$& 84.51 & 88.01 & 90.75 & 92.95 & 94.73 & 96.16 & 97.28 & 98.11 & 98.67 & 98.97 & 99 & 98.76 & 98.22 & 97.37 & 96.15 & 94.48 & 92.24 \\
					96.15 &$n=36$& 86.76 & 90.32 & 93.12 & 95.38 & 97.22 & 98.71 & 99.88 & 100.78 & 101.41 & 101.79 & 101.90 & 101.75 & 101.33 & 100.61 & 99.55 & 98.09 & 96.15 \\
					\bottomrule
			\end{tabular}}
			\label{tab:spectral_data}
			
		\end{sidewaystable}
		
	\end{proof}

	\section*{Declaration of competing interest}
	
	The authors declare that they have no known competing financial interests or personal relationships that could have appeared to influence the work reported in this paper.
	
	\section*{Data availability}
	
	No data was used for the research described in the article.

	\section*{Acknowledgements}
	\quad\quad The authors wish to thank the two anonymous referees for their valuable comments and suggestions, which have greatly helped improve the presentation of this paper.

\end{document}